\documentclass[11pt]{amsart}

\usepackage{lineno}

\usepackage{amsfonts}
\usepackage{amssymb}
\usepackage{amsmath}
\usepackage{amsthm}

\usepackage{graphicx}
\usepackage{xcolor}

\usepackage{hyperref}
\usepackage{geometry}
\theoremstyle{definition}
\newtheorem{defn}{Definition}
\newtheorem{remark}{Remark}
\theoremstyle{plain}
\newtheorem{theorem}{Theorem}
\newtheorem{lemma}[theorem]{Lemma}
\newtheorem{proposition}[theorem]{Proposition}
\newtheorem{corollary}[theorem]{Corollary}
\newtheorem*{theorem*}{Theorem}

\begin{document}
\title[Refinements of Franks' theorem]{Refinements of Franks' theorem and applications in Reeb dynamics}

\author{Hui Liu}
\address{School of Mathematics and Statistics, Wuhan University,
Wuhan 430072, Hubei, P. R. China}
\email{huiliu00031514@whu.edu.cn}

\author{Jian Wang}
\address{School of Mathematical Sciences and LPMC, Nankai University, Tianjin 300071, P. R. China}
\email{wangjian@nankai.edu.cn}

\author{Jingzhi Yan}
\address{College of Mathematics, Sichuan University, Chengdu 610065, P.R. China }
\email{jyan@scu.edu.cn}

\subjclass[2000]{37E45, 37E30}
\date{}

\begin{abstract}
In this article, we give two refinements of Franks' theorem: For  orientation and area preserving homeomorphisms of the closed or open annulus, the existence of $k$-periodic orbits ($(k,n_0)=1$) forces the existence of infinitely many periodic orbits with periods prime to $n_0$. Moreover, if $f$ is reversible,  the periodic orbits  above could be symmetric. Our improvements of Franks' theorem can be applied to Reeb dynamics and celestial mechanics,
for example, we give some precise information about the symmetries of periodic orbits found in Hofer, Wysocki and Zehnder's dichotomy
theorem when the tight 3-sphere is equipped with some additional symmetries, and also the symmetries of periodic orbits on the energy level
of H$\acute{e}$non-Heiles system in celestial mechanics.
\end{abstract}

\maketitle

\section{introduction}

In his search for periodic solutions in the restricted three body problem of celestial mechanics,  Poincar\'{e} constructed an area-preserving section map of an annulus $\mathbb{A}=\mathbb{R}/\mathbb{Z}\times [0,1]$ on the energy surface. The annulus was bounded by the so-called direct and retrograde periodic orbits. It inspired him in 1912 to formulate the following famous Poincar\'{e}-Birkhoff theorem,
 which  was later proved by Birkhoff \cite{Birkhoff13,Birkhoff26}.

\begin{theorem*}[Poincar\'{e}-Birkhoff] Every area preserving and orientation preserving homeomorphism of an annulus $\mathbb{A}$ rotating the two boundaries in opposite directions possesses at least 2 fixed points in the interior.
\end{theorem*}
The strength of this theorem is that it provides at once infinitely many periodic points if it is applied to the iterates, which leads to infinitely many periodic solutions in applications, see \cite{Birkhoff15}. More information about Poincar\'{e}-Birkhoff theorem one can refer to \cite[Chapter 6]{HZ94}. \smallskip

Under similar but much more general conditions, Franks proved the existence of infinitely many periodic orbits if the homeomorphism has two different rotation numbers \cite{Franks88,Franks92} (see Theorem \ref{thm: Franks and le calvez} below).  As a consequence, Franks proved the following celebrated theorem \cite{Franks92,Franks96}:

\begin{theorem*}[Franks]\label{thm:Franks 3.5} Suppose $F$ is an area preserving homeomorphism of the
open or closed annulus which is isotopic to the identity. If $F$ has
at least one fixed or periodic point, then $F$ must have infinitely
many interior periodic points.
\end{theorem*}

By the original work of  Hofer, Wysocki and Zehnder,  a classic method for the existence of periodic motions of Hamiltonian system with two degree is to  reduce the dynamics to the annulus-type global surface of  section (see \cite{HWZ98, HWZ03}). Under some ``good'' conditions,  a global surface of section on the energy level exists, then Poincar\'e-Birkhoff and Franks' theorems imply the existence of two or infinitely many periodic orbits.
Recently, some important progresses were made in \cite{CGHP} and \cite{CDR20} for similar results on contact three manifold, which still rely on Poincar\'e-Birkhoff and Franks' theorems. \smallskip

In this article, we will give several refinements of Poincar\'e-Birkhoff and  Franks' theorems for annulus homeomorphisms
and  applications in Reeb dynamics and celestial mechanics to get periodic orbits under some symmetric conditions.

\subsection{Periodic orbits for annulus homeomorphisms}We consider a homeomorphism $f$ of the open (resp. closed) annulus $\mathbb{A}=\mathbb{R}/\mathbb{Z}\times (0,1)\cong \{z\in\mathbb{R}^2|1< |z|< 2\}$ (resp. $\mathbb{A}=\mathbb{R}/\mathbb{Z}\times [0,1]\cong \{z\in\mathbb{R}^2|1\le |z|\le 2\}$) that is isotopic to the identity. An \emph{area} of a surface  is  a locally finite Borel  measure without atom and with total support. We say that $z$ is a {\it $k$-prime-periodic point of $f$} if $z$ is a $k$-periodic  point but not an
$l$-periodic point of $f$ for all $0<l<k$. Let $\mathrm{Per}_{k}(f)$ be the set of $k$-prime-periodic points of $f$. We obtain the following refinement of Franks' theorem.

\begin{theorem}\label{thm: refinement of Franks}
Let $f$ be a homeomorphism of closed or open annulus that is isotopic to the identity and preserves a finite area. Assume that $k,n_0\in \mathbf{N}$ which satisfy that $(k,n_0)=1$. If $f$ has a $k$-periodic point, then
  \[\sharp \left\{\bigcup_{(k\,',n_0)=1} \mathrm{Per}_{k\,'}(f) \right\}=\infty.\]
In particular, if $f$ has an odd periodic point, it has infinitely many odd periodic points.
\end{theorem}

\begin{remark}
This theorem gives a positive answer to Kang's question when he studied the planar reversible maps \cite{Kang}.
\end{remark}

\bigskip

Some classical Hamiltonian systems carry a symmetric structure, and  the global surface of sections inherit a reflection.  We are also interested in the dynamics in the symmetric case.

More precisely, let $R: (x,y)\mapsto (-x, y)$ be the reflection  of $\mathbb{R}^2$. For any $R$-invariant connected domain $\Omega\subset \mathbb{R}^2$, $R$ descends to a reflection on $\Omega$ that we still denote by $R$. A homeomorphism $f$ of $\Omega$  is called   \emph{reversible} if it satisfies
    \[f\circ R=R \circ f^{-1},\]
 and a periodic point $z\in \Omega $ is called  \emph{symmetric}  if its orbit is symmetric, i.e. $R(z)=f^k (z)$ for some $k\in\mathrm{N}$. We have the following refinement of Franks' theorem for reversible maps.

\begin{theorem}\label{thm: Franks thm for reversible map}
  Let $f$ be a homeomorphism of the closed or open annulus, that is isotopic to the identity, preserves a finite area, and is reversible. If $f$ has at least one  periodic point, then it has infinitely many symmetric periodic orbits. Moreover, given $n_0\in \mathbb{N}$,   if $f$ has a $k$-periodic point, and $(k, n_0)=1$, then
  \[\sharp \left\{  \bigcup_{(k\,',n_0)=1}\text{symmetric }  k\,' \text{-prime-periodic  points  of } f \right\}=\infty.\]
\end{theorem}

\begin{remark}
The case $n_0=2$ is just Theorem 1.3 of \cite{Kang}.
Note that even the existence of  non-symmetric periodic points  implies the existence of infinitely many symmetric periodic orbits.
\end{remark}

\bigskip

The proof of Theorem \ref{thm: Franks thm for reversible map} is based on the following two Poincar\'e-Birkhoff type results for reversible maps, which are also interesting for their own sake. The explicit definition of the rotation number is in Section \ref{subsec: rotaion number}.
\begin{theorem}\label{thm: pb for reversible map}
Let $f$ be a reversible  homeomorphism that is isotopic to the identity and $\widetilde{f}$ its lift to the universal covering space.
If there are  two positive  recurrent points with  rotation numbers $\rho^-$ and $\rho^+$ satisfying $\rho^-<0<\rho^+$, then $f$ has at least two symmetric fixed points in the interior of $\mathbb{A}$ with zero rotation number.
\end{theorem}

\begin{corollary}\label{cor: Pb for reversible map}
Let $f$ be a reversible  homeomorphism that is isotopic to the identity and $\widetilde{f}$ its lift to the universal covering space.
If there are  two positive recurrent points with different rotation numbers $\rho^-<\rho^+$, then there is a symmetric $q$-prime-periodic orbit with rotation number $\frac{p}{q}$ for all irreducible $\frac{p}{q}\in(\rho^-,\rho^+)$.
\end{corollary}

\begin{remark}
For the two results above, $f$ does not need to be area-preserving.
\end{remark}

\subsection{Applications in Reeb dynamics and celestial mechanics}

It is a classical problem in conservative dynamics to investigate the existence
of periodic motions of Hamiltonian systems restricted to energy levels.
Hofer, Wysocki and Zehnder \cite{HWZ98} used pseudoholomorphic curves in symplectizations
to study dynamically convex Reeb flows on the tight 3-sphere. They found a
special periodic orbit, which is unknotted and has Conley-Zehnder index 3, bounding
a disk-like global surface of section. One of the consequences of this global
section and Franks' theorem is that the Reeb flow admits either 2 or infinitely
many periodic orbits. This result applies to Hamiltonian dynamics on strictly
convex hypersurfaces in $\mathbb{R}^4$.

  Using the same methods in \cite{HWZ98}, Hryniewicz and Salom$\tilde{a}$o \cite{HS16}
proved a similar result for Reeb flows on $\mathbb{R}P^3 = S^3/{\mathbb{Z}_2} = L(2, 1)$ equipped
with the universally tight contact structure: if the Reeb flow is dynamically convex, then it admits an
elliptic and 2-unknotted periodic orbit which is the binding of a rational open book
decomposition. Each page is a 2-disk for the binding and constitutes a rational
global surface of section for the Reeb flow. The main motivation of this result
was to study the circular planar restricted three body problem directly on the
regularized $\mathbb{R}P^3$, without considering the usual lift to the tight 3-sphere.
Recently, Schneider\cite{Sch20} generalizes results in \cite{HWZ98} and \cite{HS16}
to lens spaces $L(p, 1), p > 1$, equipped the standard tight contact structure:
if the Reeb flow on such a contact manifold is dynamically convex, then it admits a closed Reeb orbit $P$ which
is $p$-unknotted and bounds a rational disk-like global surface of section. Moreover,
the Conley-Zehnder index of the $p$-th iterate of $P$ is 3. In fact, $P$ is the binding
of a rational open book decomposition whose pages are rational global surfaces of
section, and Kim \cite{Kim} further considers
the existence of invariant disk-like global surface of section on lens spaces $L(p, 1), p > 1$, equipped with a
tight contact form and an anti-contact involution.

In the last section, combining our refinements of Franks's theorem with the works in  \cite{HS16}, \cite{Sch20}
and \cite{Kim}, we give some precise information about the symmetries of periodic orbits found in Hofer, Wysocki and Zehnder's dichotomy
theorem when the tight 3-sphere is equipped with some additional symmetries, and also the symmetries of periodic orbits on the energy level
of H$\acute{e}$non-Heiles system in celestial mechanics.

\subsection{Organization of the paper} In Section 2, we will give several definitions and preliminary results. In Section 3, we prove a refinement of Franks' Theorem: Theorem \ref{thm: refinement of Franks}. In Section 4, we prove a Poincare-Birkhoof type theorem for reversible homeomorphisms: Theorem \ref{thm: pb for reversible map}, and get Corollary \ref{cor: Pb for reversible map} as a direct corollary. In Section 5, we prove a refinement of Franks' Theorem for reversible homeomorphisms: Theorem \ref{thm: Franks thm for reversible map}. In Section 6, we give the applications of Theorem \ref{thm: refinement of Franks} and Theorem \ref{thm: Franks thm for reversible map} to Reeb dynamics and celestial mechanics.\bigskip

\noindent\textbf{Acknowledgements.} We would like to thank Umberto Hryniewicz, Patrice
Le Calvez and Yiming Long for their helpful conversations and
comments. Hui Liu was partially supported by NSFC (Nos. 12022111, 11771341) and
the Fundamental Research Funds for the Central Universities (No. 2042021kf1059).
Jian Wang was partially supported by NSFC (Nos. 12071231, 11971246). Jingzhi Yan was partially supported by NSFC(Nos. 11901409, 11831012) \bigskip


\section{Preliminaries}

\subsection{Rotation number}\label{subsec: rotaion number}
 We denote  by $\mathbb{A}$ the open  (resp.  the closed) annulus unless an explicit mention, i.e.  $\mathbb{A}=\mathbb{R}/\mathbb{Z}\times (0,1)\cong \{z\in\mathbb{R}^2|1< |z|< 2\}$ ( resp. $\mathbb{A}=\mathbb{R}/\mathbb{Z}\times [0,1]\cong \{z\in\mathbb{R}^2|1\le |z|\le 2\}$),
 by $\pi$ the covering map of the  annulus
\begin{eqnarray*}
\pi\,:\, \mathbb{R}\times(0,1)\quad (\mathrm{resp.}\quad\mathbb{R}\times[0,1])&\rightarrow& \mathbb{A}\\
(x,y)&\mapsto&(x+\mathbb{Z},y),
\end{eqnarray*}
and by  $T$ the generator of the covering transformation group
\begin{eqnarray*}
T\,:\, \mathbb{R}\times(0,1)\quad (\mathrm{resp.}\quad\mathbb{R}\times[0,1])&\rightarrow& \mathbb{R}\times(0,1)\quad (\mathrm{resp.}\quad\mathbb{R}\times[0,1]) \\
(x,y)&\mapsto&(x+1,y).
\end{eqnarray*}
Coordinates are denoted as $z\in\mathbb{A}$ and $\tilde{z}$ in the covering space. Homeomorphisms of $\mathbb{A}$ are denoted by $f$, and their lifts to the covering space are denoted by $\widetilde{f}$.

Consider the homeomorphism $f$ of $\mathbb{A}$ that is isotopic to the identity.  We say that a positively recurrent point $z$ has a \emph{rotation number} $\rho(\widetilde{f}, z)\in \mathbb{R}$ for a lift $\widetilde{f}$ of $f$ to the universal covering space of $\mathbb{A}$, if for every subsequence $\{f^{n_k}(z)\}_{k\geq 0}$ of $\{f^n(z)\}_{n\geq 0}$ which converges to $z$, we have
\[\lim_{k\rightarrow+\infty}\frac{p_1\circ \widetilde{f}^{n_k}(\widetilde{z})-p_1(\widetilde{z})}{n_k}=\rho(\widetilde{f}, z)\]
where $\widetilde{z}\in \pi^{-1}(z)$ is a lift of $z$ and $p_1$ is the standard projection to the first coordinate.
The rotation number is stable by conjugacy  (see \cite{Lecalvez01}).
 In particular, the rotation number $\rho(\widetilde{f}, z)$ always exists and is rational when $z$ is a fixed or periodic point of $f$.

  Let $\mathrm{Rec}^+(f)$ be the set of positively recurrent points of $f$. We denote the set of rotation numbers of positively recurrent points of $f$ as $\mathrm{Rot}(\widetilde{f})$.
A positively recurrent point of $f$ is also a positively recurrent point of $f^q$ for all $q\in \mathbb{N}$ \cite[Appendix]{Wang14}. By the definition of the rotation number, we easily get  the following elementary properties.
\begin{enumerate}\label{prop:ROT}
  \item[1.] $\rho(T^k\circ\widetilde{f}, z)=\rho(\widetilde{f}, z)+k$, and hence $\mathrm{Rot}(T^k\circ \widetilde{f})=\mathrm{Rot}(\widetilde{f})+k$ for every $k\in \mathbb{Z}$;
  \item[2.] $\rho(\widetilde{f}^q, z)=q\rho(\widetilde{f}, z)$, and hence $\mathrm{Rot}(\widetilde{f}^q)=q\mathrm{Rot}(\widetilde{f})$ for every $q\in \mathbb{N}$.
\end{enumerate}

We call a simple closed curve  in $\mathbb{A}$ an \emph{essential circle} if  it is not null-homotopic. We say that $f$ satisfies the \emph{intersection property} if  any essential circle in $\mathbb{A}$ meets its image by $f$. It is easy to see that a homeomorphism $f$ that preserves a finite area satisfies the intersection property.

We need the following Theorem due to Franks \cite{Franks88} when $\mathbb{A}$ is the closed annulus and $f$ has no wandering point,  and  improved by Le Calevez \cite{Lecalvez05} (see also \cite{Wang14}) when $\mathbb{A}$ is the open annulus and $f$ satisfies the intersection property:

\begin{theorem}\label{thm: Franks and le calvez}Let $f$ be a homeomorphism of $\mathbb{A}$ that is isotopic to the identity and satisfies the intersection condition,  and $\widetilde{f}$  one of its lifts to the universal covering space. Suppose that there exist two recurrent points $z_{1}$ and $z_{2}$ such that
$-\infty\leq \rho(\widetilde{f}, z_{1})<\rho(\widetilde{f}, z_{2})\leq +\infty$. Then for
any rational number $p/q\in (\rho(\widetilde{f}, z_{1}),\rho(\widetilde{f}, z_{2}))$
written in an irreducible way, there exists a q-prime-periodic point with rotation number $p/q$.
\end{theorem}

\subsection{Transverse foliation and maximal isotopy}
Let $M$ be an oriented surface, and $\mathcal{F}$ an oriented topological foliation on $M$ whose leaves are oriented curves. We say that a path is \emph{positively transverse} to $\mathcal{F}$, if it meets the leaves of $\mathcal{F}$ locally from left to right.
Let $f$ be a homeomorphism on $M$, and $I=(f_t)_{t\in[0,1]}$  an identity isotopy of $f$, i.e. an isotopy joining the identity to $f$.    We say that an oriented foliation $\mathcal{F}$ (without singularity)  is a \emph{transverse foliation}  of $I$ if for every $z\in M$, there is a path  that is homotopic to the trajectory $t\rightarrow f_t(z)$ of $z$ along $I$ with the end points fixed and  is positively transverse to   $\mathcal{F}$.
If $f$ does not have any contractible fixed point associated to $I$, i.e. a fixed point of $f$ whose trajectory along $I$ is null homotopic in $M$, Le Calvez  proved the existence of transverse foliation \cite[Theorem 1.3]{Lecalvez05}.

For the case with contractible fixed point, we call  $K\subset \mathrm{Fix}(f)$ \emph{unlinked} if there is an identity isotopy  $I=(f_t)$ of $f$ such that $K\subset \mathrm{Fix}(I)=\cap_{t\in[0,1]}\mathrm{Fix}(f_t)$. We call an identity isotopy \emph{maximal}, if $\mathrm{Fix}(I)$ is maximal for including among all unlinked sets.  Moreover, if $I$ is a maximal isotopy, $f|_{M\setminus \mathrm{Fix}(I)}$ does not have any contractible fixed point. We consider a singular foliation $\mathcal{F}$ and call it  a \emph{transverse foliation} of $I$, if the set of singularities $\mathrm{Sing}(\mathcal{F})$ is equal to the fixed point set $\mathrm{Fix}(I)$ of $I$, and if $\mathcal{F}|_{M\setminus\mathrm{Sing}(\mathcal{F})}$ is transverse to $I|_{M\setminus \mathrm{Fix}(I)}$. Combine Le Calvez's result and the following theorem, we always get the existence of  a maximal isotopy $I$ and a transverse foliation $\mathcal{F}$.

\begin{theorem}\cite[Corollary 1.2]{BCLR}\label{thm: BCLR}
Let $f$ be a homeomorphism of $M$ and $K$ an unlinked set. Then,  there is a maximal isotopy $I$ such that $K\subset \mathrm{Fix}(I)$.
\end{theorem}

In particulary, for  area-preserving homeomorphisms of the sphere, the dynamics of the transverse foliation is quite simple: there is neither a closed leaf nor a leaf from and toward  the same singularity,  and hence every leaf joins one singularity to another singularity.

\subsection{Reversible maps and symmetric periodic points}
We consider the reflection $R$ and a reversible map $f$ of the $R$-invariant domain $\Omega\subset \mathbb{R}^2$.  We denote by $R_{\Omega}$ the restriction of $R$ to $\Omega$. The fixed points set $\mathrm{Fix}(R_{\Omega})$  is composed of disjoint intervals $Y_i$, i.e. $\mathrm{Fix}(R_{\Omega})=\sqcup_i Y_i$.    Kang \cite{Kang} gave a necessary and sufficient condition  for the existence of a symmetric fixed point on any interval $Y_i$.

\begin{theorem} \cite[Theorem 1.5]{Kang}\label{thm:Kang}
Let $f:\Omega\to \Omega$ be a reversible map isotopic to the identity. Then $f$ has a symmetric fixed point on $Y_i$ if and only if $f(Y_i)\cap Y_i\ne \emptyset$. And there exists an interior symmetric fixed point on $Y_i$ if and only if $f(Y_i)\cap Y_i\cap \mathrm{Int}(\Omega)\ne \emptyset$.
\end{theorem}

  In particular, the annulus $\mathbb{A}$ is invariant by $R$, and the fixed points set is composed of two disjoint interval $Y_1$ and $Y_2$. Choose a lift $\widetilde{Y}$ of $Y=Y_1$ (or $Y_2$) in the universal covering space of $\mathbb{A}$, there is a lift $\widetilde{R}$ of $R_{\mathbb{A}}$ that fixes $\widetilde{Y}$. Moreover, $\widetilde{R}$ is a reflection with axis $\widetilde{Y}$. One should note the following two facts:
\begin{itemize}
\item[-] $T\circ \widetilde{R}=\widetilde{R}\circ T^{-1}$.
\item[-] Let $f$ be a reversible map and $\widetilde{f}$ a lift of $f$. Then, $\widetilde{f}\circ\widetilde{R}=\widetilde{R}\circ\widetilde{f}^{-1}$.
\end{itemize}
The first fact is obvious. For the second fact, since $\pi\circ \widetilde{f}\circ\widetilde{R}=f\circ\pi\circ\widetilde{R}=f\circ R_{\mathbb{A}}$, and $\pi\circ\widetilde{R}\circ\widetilde{f}^{-1}=R_\mathbb{A}\circ\pi\circ\widetilde{f}^{-1}=R_\mathbb{A} \circ f^{-1}$,
$\widetilde{f}\circ\widetilde{R}$ and $\widetilde{R}\circ\widetilde{f}^{-1}$ are both lifts of $f\circ R_\mathbb{A}=R_\mathbb{A}\circ f^{-1}$. Furthermore, we suppose that $\widetilde{f}\circ\widetilde{R}\circ\widetilde{f}\circ\widetilde{R}=T^l$ for some $l\in\mathbb{Z}$. Then we have
$$T^{2l}=\widetilde{f}\circ\widetilde{R}\circ\widetilde{f}\circ\widetilde{R}\circ\widetilde{f}\circ\widetilde{R}\circ\widetilde{f}\circ\widetilde{R}=\widetilde{f}\circ\widetilde{R}\circ T^l\circ\widetilde{f}\circ\widetilde{R}=\widetilde{f}\circ\widetilde{R}\circ T^{2l}\circ \widetilde{R}\circ\widetilde{f}^{-1}=T^{-2l}\circ \widetilde{f}\circ\widetilde{R}\circ\widetilde{R}\circ\widetilde{f}^{-1}=T^{-2l}.$$
It implies that $l=0$, which completes the proof.

\section{Proof of Theorem \ref{thm: refinement of Franks}}
In this section, $f$ is  a homeomorphism of the closed or open annulus that is isotopic to the identity and preserves a finite area. Suppose that $f$ has a $k$-periodic point and $(k,n_0)=1$, we will prove
 \[\sharp \left\{\bigcup_{(k\,',n_0)=1} \mathrm{Per}_{k\,'}(f) \right\}=\infty.\]
The rest part of Theorem \ref{thm: refinement of Franks} is a direct corollary.

\begin{proof}[Proof of  Theorem \ref{thm: refinement of Franks}]
If $\sharp\mathrm{Per}_k(f)=+\infty$, we have nothing to do. Hence we assume that $\sharp\mathrm{Per}_k(f)<+\infty$. Let $g=f^k$. If $\sharp\mathrm{Fix}(g)=+\infty$, then there exists $k'\,|\,k$ with $(k',n_0)=1$ such that $\sharp\mathrm{Per}_{k'}(f)=+\infty$. We are done. Therefore, we can assume that $1\leq\sharp\mathrm{Fix}(g)<+\infty$.

We consider the rotation number $\rho(\widetilde{f}, z)$, where  $\widetilde{f}$ is a lift of $f$ and $z$ is a positively recurrent point of $f$.
We will divided the proof into the following two cases:

\begin{itemize}
\item[i)]Suppose  there exists $z_1,z_2$ with rotation numbers $\rho(\widetilde{f}, z_1)< \rho(\widetilde{f}, z_2)$.
    By Theorem \ref{thm: Franks and le calvez}, $f$ has a $q$-periodic orbit with rotation number $p/q$ for all irreducible $\frac{p}{q}\in( \rho(\widetilde{f}, z_1), \rho(\widetilde{f}, z_2) )$. We get the result by choosing irreducible $\frac{p}{q}\in( \rho(\widetilde{f}, z_1), \rho(\widetilde{f}, z_2) )$ such that $(q,n_0)=1$.

\item[ii)]   Suppose that $\rho(\widetilde{f}, z)$ is a constant   for all positively recurrent $z$. In particular, $\rho(\widetilde{f}, z)=\frac{\ell}{k}$ with $\ell\in \mathbb{Z}$, for $z\in \mathrm{Per}_k(f)$.

 If $\mathbb{A}$ is a closed annulus, we shrink each boundary to a point and get a sphere; while $\mathbb{A}$ is an open annulus, we consider the end  compactifications and also get a sphere. Denote the sphere by $S^2$, and the two points not in $\mathbb{A}$ by $N$ and $S$. Moreover, $g$ induces an area-preserving homeomorphism of $S^2$ that fixes both $N$ and $S$. We still denote by $g$ the induced homeomorphism of $S^2$

  By Theorem \ref{thm: BCLR}, there is a maximal isotopy $I=(g_t)_{t\in [0,1]}$ that fixes both $N$ and $S$.
 We consider the lift $\widetilde{g}_t$ of $g_t$ to the universal covering space of $\mathbb{A}$ and call $\widetilde{g}=\widetilde{g}_1$ the lift of $g$ associated to $I$. By composing a rotation of the sphere to the isotopy if necessary, we can suppose $\rho(\widetilde{g},z)=0$ for all  positively recurrent $z\in\mathbb{A}$.

 We will first prove  if $I$ fixes only two points $N$ and $S$, then there is a positively recurrent point $z\in\mathbb{A}$ such that $\rho(\widetilde{g},z)\ne 0$, and hence $I$ has at least three fixed points. In this paragraph, we suppose $I$ has only two fixed points $N$ and $S$ and  consider the transverse foliation $\mathcal{F}$  of $I$. The transverse foliation $\mathcal{F}$ has exactly two singularities $N$ and $S$ and hence all the leaves are topological lines from one singularity to the other singularity.  We choose one leaf $\gamma$ and one of its lifts $\widetilde{\gamma}$.   Note that the oriented curve $\widetilde{\gamma}$ separates the covering space into two parts, and  $\widetilde{g}$ maps the  part on the right of $\widetilde{\gamma}$  to a proper subset of itself.  we can choose a small disk $U$ near $\gamma$ and $\widetilde{U}$ the lift of $U$ near $\widetilde{\gamma}$, such that $\widetilde{U}$ is on the left of $\widetilde{\gamma}$ and $\widetilde{g}(\widetilde{U})$ is on the right of $\widetilde{\gamma}$. We can also assume  the diameter of  $p_1(\widetilde{U})$ is smaller than $\frac{1}{2}$  by choosing $U$ small enough, where $p_1$ is the projection to the first factor. We suppose $p_1(\widetilde{z})-p_1(\widetilde{z'})>\frac{1}{2}$ for all $\widetilde{z}\in \widetilde{U}$ and $\widetilde{z'}\in \widetilde{U}'$, where $\widetilde{U}'$ is any lift of $U$ on the right of $\widetilde{\gamma}$ (The proof in the case $p_1(\widetilde{z})-p_1(\widetilde{z'})<-\frac{1}{2}$ is similar).
  By Poincar\'e's recurrent theorem, almost all points in $U$ are recurrent by $g$.  We denote by $\Phi$  the first return map on  $U\cap \mathrm{Rec}^+(g)$ and by $\tau$ the first return time.
 Let $m(z)=p_1\circ\widetilde{g}^{\tau (z)}(\widetilde{z})-p_1(\widetilde{z})$, for $z\in U\cap \mathrm{Rec}^+(g)$.Then $m(z)> \frac{1}{2}$, because the positive orbit of $\widetilde{z}$ by $\widetilde{ g}$ will on the right of $\widetilde{\gamma}$.
  Let $\tau_n(z)=\sum\limits_{i=0}^{n-1}\tau(\Phi^i(z))$ and $m_n(z)=\sum\limits_{i=0}^{n-1}m(\Phi^i(z))$. Then, $\frac{m_n(x)}{\tau_n(x)}\to \rho(\widetilde{g},z)$ as $n\to+\infty$, if the rotation number exists.
 Note that
  \[\frac{m_n(x)}{\tau_n(x)}=\frac{m_n(x)}{n}\frac{n}{\tau_n(x)},\quad  \frac{m_n(x)}{n}>\frac{1}{2}.\]
   By  Kac's Theorem \cite{Wright}, $\tau\in L^1(U,\mathbb{R})$. Then, $\frac{\tau_n}{n}$ is convergent a.e. by Birkhoff's ergordic theorem.
    Therefore, $\rho(\widetilde{g},z)>0$  a.e. $z\in U$.

 Let $\mathcal{F}$ be the transverse foliation of $I$, and choose a leaf $\gamma$ joining two singularities $z_1$ and $z_2$. We consider the annulus $S^2\setminus\{z_1, z_2\}$, the restricted homeomorphism $g$  on the new annulus, and the lift $\widetilde{g}$ of $g$ to the universal covering space of the new annulus associate to $I$. Since $I$ has at least three fixed points, there is a fixed point $z$ such that $\rho(\widetilde{g},z)=0$. By a similar discussion as in the previous paragraph, there is a recurrent point $z'$ such that $\rho(\widetilde{g},z')\ne 0$. By Theorem \ref{thm: Franks and le calvez}, for all irreducible $\frac{p}{q}$ between $0$ and $\rho(\widetilde{g},z')\ne 0$, $g$ has a $q$-prime-periodic orbit with rotation number $\frac{p}{q}$. By choosing $(q,n_0)=1$, we finish the proof. \qedhere.
 \end{itemize}
\end{proof}

\begin{remark}
  The idea to get  non-zero rotation numbers by the transverse foliation is  from P. Le Calvez (see \cite{Lecalvez05} for example).
\end{remark}

\section{A Poincar\'e-Birkhoff type theorem for reversible homeomorphism}

In this section, we consider reversible  homeomorphism $f$  of the closed or open annulus that is isotopic to the identity. Let $\widetilde{f}$ be a lift of $f$ to the universal covering space.  We will prove a Poincar\'e-Birkhoff type theorem: Theorem \ref{thm: pb for reversible map}.  Corrollary \ref{cor: Pb for reversible map} is its direct corollary by considering  $f^q$ and the lift  $T^{-p}\circ\widetilde{f}^q$, where $T$ is the generater of the covering transformation group as in Section \ref{subsec: rotaion number}.

\begin{proof}[Proof of Theorem \ref{thm: pb for reversible map}]
Let $Y_1$ and $Y_2$ be the two interval of the fixed points set of the reflection $R_{\mathbb{A}}$.
   For any lift $\widetilde{Y}_i$ of $Y_i$, $i=1,2$, we will prove by contradiction that $\widetilde{f}(\widetilde{Y}_i)\cap \widetilde{Y}_i\cap\mathbb{R}\times(0,1)\ne \emptyset$.
 Suppose that $\widetilde{f}(\widetilde{Y}_i)\cap \widetilde{Y}_i\cap\mathbb{R}\times(0,1)= \emptyset$,  then $\widetilde{f}(\widetilde{Y_i})$ is on the left or on the right of $\widetilde{Y}_i$. Suppose that it is on the right, the other case can be treated similarly.
 Then,
\[ p_1\circ\widetilde{f}^n(\widetilde{z})-p_1(\widetilde{z})\ge -1, \text{for all } n\ge 1.\]
 and hence, $\rho(\widetilde{f}, \widetilde{z})\ge 0$ for all recurrent $z$. We get a contradiction.

 We choose $\widetilde{Y}_i$ and a lift $\widetilde{R}$ of $R_\mathbb{A}$ that fixes every point of $\widetilde{Y}_i$. Then $\widetilde{f}\circ\widetilde{R}=\widetilde{R}\circ \widetilde{f}^{-1}$.
We get the existence of a fixed point of $\widetilde{f}$ on $\widetilde{Y}_i\cap \mathbb{R}\times (0,1)$  by Theorem \ref{thm:Kang}.
Therefore, $f$ have a fixed point in $Y_1$ and a fixed point in $Y_2$, both fixed points are in the interior of $\mathbb{A}$ and have zero rotation numbers.
\end{proof}

\section{Proof of Theorem \ref{thm: Franks thm for reversible map}}

In this section, we suppose that $f$ is a reversible homeomorphism of the closed or the open annulus that is isotopic to the identity and preserves a finite area. We will prove the following proposition and then get Theorem \ref{thm: Franks thm for reversible map} as a corollary.

\begin{proposition}\label{prop: reversible-fixed point to periodic point}
  Suppose that $f$ has a  fixed point, and has at most finitely many symetric fixed points. Then, there exists $\rho'\ne 0$ such that for all irreducible $\frac{p}{q}$ between $0$ and $\rho'$, $f$ has symmetric $q$-prime-periodic points.
\end{proposition}

\begin{proof}
Fix a lift $\widetilde{f}$ of $f$ that has a fixed point. The rotation number of any fixed point, whose lifts are fixed points  of $\widetilde{f}$, is zero.  Our idea is to prove the existence of positively recurrent points with different  rotation numbers, which guarantees the existence of  symmetric periodic orbits by Corollary \ref{cor: Pb for reversible map}.

 If there is a positively recurrent point (in the interior or on the boundary) with  non-zero rotation number, then there is a   symmetric $q$-prime-periodic orbit with rotation number $\frac{p}{q}$ for all irreducible $\frac{p}{q}$ between the two different rotation number by Corollary \ref{cor: Pb for reversible map}.

Now, we suppose that all the fixed point of $f$ has  rotation number zero. Recall that almost all recurrent points of $f$ have rotation numbers \cite[Theorem 1]{Lecalvez01}.  We will deduce the existence of recurrent points with different rotation numbers in two cases.

\begin{itemize}
\item[i)] Suppose that there is a connected component of the  fixed point set of $R_{\mathbb{A}}$ such that $Y\cap f(Y)=\emptyset$, where $Y$ is the component itself when $\mathbb{A}$ is  an open annulus, and is the maximal open interval in  this component when $\mathbb{A}$ is a closed  annulus.
    We choose a lift $\widetilde{Y}$ of $Y$. Then, $\widetilde{f}(\widetilde{Y})\cap\widetilde{Y}=\emptyset$. We suppose that $\widetilde{f}(\widetilde{Y})$ is on the right of $\widetilde{Y}$, we will deduce the existence of a positive rotation number. (In the case $\widetilde{f}(\widetilde{Y})$ is on the left of $\widetilde{Y}$, the proof is similar, but we can get  the existence of a negative rotation number.)

    Choose a small disk $U$ near $Y$ (with distance $<\frac{1}{4}$) such that one of its lift $\widetilde{U}$ is on the left of $\widetilde{Y}$ and $\widetilde{f}(\widetilde{U})$ is on the right of $\widetilde{Y}$. Because $\widetilde{f}(\widetilde{Y})$ is on the right of $\widetilde{Y}$, $\widetilde{f}(\mathrm{Right(\widetilde{Y})})\subset\mathrm{Right}(\widetilde{Y})$ and $\widetilde{f}(\mathrm{Right(T^k\widetilde{Y})})\subset\mathrm{Right}(T^k\widetilde{Y})$ for all $k\in\mathbb{Z}$.
    Recall that $f$ preserves a finite area. By Poincar\'{e}'s recurrence theorem,  almost all point are recurrent. We denote by $\Phi$  the first return map on  $U\cap \mathrm{Rec}^+(f)$ and by $\tau$ the first return time.
    Let $m(z)=p_1\circ\widetilde{f}^{\tau (z)}(\widetilde{z})-p_1(\widetilde{z})$, for $z\in U\cap \mathrm{Rec}^+(f)$.
    Then $m(z)> \frac{1}{2}$, because the positive orbit of $\widetilde{z}$ by $\widetilde{ f}$ will never go to the left of $\widetilde{Y}$.
    Let $\tau_n(z)=\sum_{i=0}^{n-1}\tau(\Phi^i(z))$ and $m_n(z)=\sum_{i=0}^{n-1}m(\Phi^i(z))$. Then, $\frac{m_n(x)}{\tau_n(x)}\to \rho(\widetilde{f},z)$ as $n\to+\infty$, if the rotation number exists.
    Note that
    \[\frac{m_n(x)}{\tau_n(x)}=\frac{m_n(x)}{n}\frac{n}{\tau_n(x)},\quad  \frac{m_n(x)}{n}>\frac{1}{2}.\]
    By  Kac's Theorem \cite{Wright}, $\tau\in L^1(U,\mathbb{R})$. Then, $\frac{\tau_n}{n}$ is convergent a.e. by Birkhoff's ergordic theorem.
    Therefore, $\rho(\widetilde{f},z)>0$ for a.e. $z\in U$.

\item[ii)] Suppose that $Y_i\cap f(Y_i)\cap \mathrm{Int}(\mathbb{A})\ne \emptyset$, where $\mathrm{Fix}(R_{\mathbb{A}})=Y_1\sqcup Y_2$. By Theorem \ref{thm:Kang}, $f$ have symmetric  fixed points in the interior of $\mathbb{A}$.
 Let $z_1=(0,y)$ be a fixe point  of $f$, such that the symmetry interval $Y$ between $z_1$ and the outer-boundary contains no other fixed point. By removing $z_1$ and contracting the inner boundary to one point, we get an annulus and the induced reflection and  reversible map.
By repeating the discussion in the previous case, we get the existence of different rotation numbers and finish the proof.\qedhere
    \end{itemize}
\end{proof}

\begin{proof}[Proof of Theorem \ref{thm: Franks thm for reversible map}]
We consider  $g=f^k$. A symmetric fixed point of $g$  is a symmetric periodic point of $f$ with period $k'\mid k$, and hence $(k',n_0)=1$. A symmetric $q$-periodic point of $g$ on $Y$ induces a symmetric periodic orbit of $f$ with period $k'\mid kq$. Note that $k'\nmid n_0$ if $q\nmid n_0$. We get the result as a corollary.
\end{proof}

\section{Applications in Reeb dynamics and celestial mechanics}

Firstly, we review some basic concepts in contact geometry, cf., \cite{Sch20} and \cite{Kim}.
Recall that a contact three-manifold is a three-manifold $M$ equipped with a maximally non-integrable
hyperplane distribution $\xi$, called a contact structure. If $\xi$ is co-oriented, then there
exists a global one-form $\lambda$ on M, called a contact form (defining $\xi$), such that
$\xi = ker \lambda$. Note that if $\lambda$ is a contact form, then $\lambda \wedge d\lambda$ is non-vanishing
and that for every smooth non-vanishing function $f$ on $M$, the one-form $f\lambda$ is a contact
form defining the same contact structure as $\lambda$.

A contact structure $\xi$ on M is called overtwisted if there exists an embedded
disk $D\hookrightarrow M$ such that $T_z(\partial D) \subset \xi_z$ and $T_zD \neq \xi_z$ for all $z\in \partial D$.
If such a disk does not exist then the contact structure is called tight.

By abuse of terminology, we also
call the pair $(M, \lambda)$ a contact manifold. We are interested in the dynamics of the
Reeb vector field $X_\lambda$ of $\lambda$ uniquely characterized by the equations
\[\lambda(X_\lambda) = 1\quad \text{and}\quad \imath_{X_\lambda} d \lambda = 0.\]
A periodic Reeb orbit will be denoted by $P = (x, T )$, where $x: \mathbb{R} \rightarrow M$ solves the
differential equation $\dot{x}= X_\lambda\circ x$ and $T >0$ is a period. It is said to be simply covered
if $T$ is the minimal period, namely, if $k \in \mathbb{N}$ is such that $T/k$ is a period of $x$, then
$k = 1$.We denote by $\mathcal {P}(\lambda)$ the set of equivalence classes of periodic orbits
of $\lambda$ with the identification
\[P = (x, T) \sim Q = (y, T^\prime) \Leftrightarrow T = T^\prime\quad \text{and}\quad x(\mathbb{R}) = y(\mathbb{R}).\]
We say that $P = (x, T)$ is contractible if the loop
\[\begin{array}{cccc}
   x_T : &\mathbb{R}/ \mathbb{Z} &\rightarrow& M \\
   & t  &\mapsto& x(Tt)
  \end{array}\]
is contractible on $M$. If the first Chern class $c_1(\xi)$ vanishes on $\pi_2(M)$, then every
contractible periodic orbit $P\in\mathcal {P}(\lambda)$ has a well-defined Conley-Zehnder index
$\mu_{CZ}(P)\in \mathbb{Z}$. As in \cite{HWZ98}, a contact form $\lambda$ on a smooth
closed 3-manifold $M$ is called {\it dynamically convex} if $c_1(ker\lambda)$ vanishes on $\pi_2(M)$
and every contractible periodic orbit $P \in \mathcal {P}(\lambda)$ satisfies $\mu_{CZ}(P)\geq 3$,
and by \cite{HWZ99}, if $\lambda$ is a dynamically convex
contact form on a closed 3-manifold $M$, then $\pi_2(M)$ vanishes and the contact structure $ker \lambda$ is tight.

A knot $K \hookrightarrow M$ is called $k$-unknotted, for some $k\in \mathbb{N}$, if there exists an
immersion $u : \mathbb{D} \rightarrow M$, such that $u_{\mathbb{D}\setminus \partial \mathbb{D}}$ is an embedding and
$u_{\partial \mathbb{D}} : \partial \mathbb{D}\rightarrow K$ is a $k$-covering map. The map $u$ is called a $k$-disk for $K$. If $K$ is
oriented then we say that $u$ induces the same orientation as $K$ if $u_{\partial \mathbb{D}}$ preserves orientation, where
$\partial \mathbb{D}$ has the counter-clockwise orientation. If the $k$-unknotted $K$ is transverse to $\xi$,
then $K$ is oriented by $\lambda$ and there exists a well-defined rational self-linking number
$sl(K, u) \in\mathbb{Q}$, computed with respect to a $k$-disk $u$ for $K$. If the first Chern
class $c_1(\xi)$ vanishes on $\pi_2(M)$, then $sl(K) = sl(K, u)$ does not depend on the choice
of $u$.

\begin{defn}
Let $\lambda$ be a defining contact form for a closed contact 3-manifold
$(M, \xi)$. Let $K \hookrightarrow M$ be a $k$-unknotted closed Reeb orbit of $\lambda$. A rational disk-like
global surface of section bounded by $K$ is a $k$-disk $u : \mathbb{D} \rightarrow M$ for $K = u(\partial \mathbb{D})$ so that
$u(\mathbb{D}\setminus \partial \mathbb{D})$ is transverse to $X_\lambda$, and every Reeb trajectory in $M\setminus K$ hits $u({\mathbb{D}\setminus \partial \mathbb{D}})$ infinitely many times in the past and in the future. In particular, the Reeb
flow of $\lambda$ is encoded in the corresponding first return map
$\psi : u(\mathbb{D}\setminus \partial \mathbb{D}) \rightarrow u(\mathbb{D}\setminus \partial \mathbb{D})$.
\end{defn}

Note that in \cite{Kim}, we also call this $k$-disk $u$ in the above definition a {\it k-rational disk-like global surface of section.}

A smooth involution $\varrho$ defined on a contact manifold $(M, \lambda)$ is said to be {\it anti-contact} if $\varrho^*\lambda = -\lambda$.
In this case, the triple $(M, \lambda, \varrho)$ is called a {\it real contact
manifold.} In this case, we call a global surface of section $\Sigma$ {\it invariant} if
$\varrho(\Sigma) = \Sigma$.

Let $(x_1, x_2, y_1, y_2)$ be coordinates in $\mathbb{R}^4$, equip $\mathbb{R}^4$ with the standard symplectic form
\[\omega_0=\sum_{i=1}^2dy_i\wedge dx_i.\]
The Liouville form
\[\lambda_0 =\frac{1}{2}\sum_{i=1}^2(y_idx_i-x_idy_i),\]
is a primitive of $\omega_0$ and restricted to a contact form on the 3-sphere
\[S^3 = \{|x_1|^2 + |x_2|^2+ |y_1|^2 + |y_2|^2= 1\}.\]
The contact structure $\xi_0= \mathrm{ker} \lambda_0$ is called the
standard tight contact structure on $S^3$. It is well known that it is the unique tight contact structure on $S^3$ up
to diffeomorphisms. We call the pair $(S^3, \xi_0)$ the {\it tight three-sphere.} A contact
form defining $\xi_0$ is called a {\it tight contact form}. Let $\lambda = f\lambda_0$ be a tight contact
form. Then a smooth involution $\varrho$ is anti-contact if and only if $f \circ\varrho = f$ and
$\varrho^*\lambda_0 = -\lambda_0$. The triple $(S^3, \lambda,\varrho)$ is called a {\it real tight three-sphere} (with respect to
$\varrho$).

Given relatively prime integers $p\geq q \geq 1$, there exists a free
action of $\mathbb{Z}_p := \mathbb{Z}/{p\mathbb{Z}}$ on $(S^3, \xi_0)$ generated by the $p$-periodic contactomorphism $g_{p,q} : S^3 \rightarrow S^3$
\[g_{p,q}(z_1=x_1+iy_1, z_2=x_2+iy_2) = (e^{2\pi i/p}z_1, e^{2\pi iq/p}z_2),\]
via the identification $\mathbb{C}^2=\mathbb{R}^4$.
The orbit space $L(p, q) := S^3/{\mathbb{Z}_p}$ is called a {\it lens space}. The contact structure $\xi_0$ descends to a tight contact form
on $L(p, q)$, still denoted $\xi_0$. It is called the standard tight contact structure on
$L(p, q)$.
When we consider a real tight three-sphere $(S^3, \lambda,\varrho)$ endowed with a contactomorphism $g_{p,q}$ as above,
where $\varrho$ is an anti-contact involution of the form
$(z_1, z_2) \mapsto (e^{i\vartheta_1}\bar{z}_1, e^{i\vartheta_2}\bar{z}_2), \vartheta_1, \vartheta_2\in\mathbb{R}$.
Then we have
\begin{eqnarray}
g_{p,q}\circ\varrho\circ g_{p,q}=\varrho.\nonumber
\end{eqnarray}
For each $j = 0, 1, \ldots, p-1$, the smooth map $\varrho_j := g_{p,q}^j\circ\varrho$ is an anti-contact involution
with $\varrho_0 = \varrho$. A periodic orbit $P = (x, T )$ is said to be {\it $(\varrho, g_{p,q})$-symmetric} if $\varrho_j(x(\mathbb{R})) =
x(\mathbb{R})$ for all $j$. The anti-contact involution $\varrho$ descends
to an anti-contact involution $\bar{\varrho}$ on $(L(p, q), \lambda)$. We call the triple $(L(p, q), \lambda, \bar{\varrho})$ a
real universally tight lens space.

\begin{defn}
Let $\lambda$ be a contact form on $L(p, q)$ inducing the standard tight
contact structure $\xi_0$. A rational open book decomposition with disk-like pages
and binding orbit $K$ is a pair $(\pi,K)$ formed by a $p$-unknotted closed Reeb orbit
$K \hookrightarrow L(p, q)$ and a smooth fibration $\pi: M \setminus K \rightarrow S^1$ so that the closure of each
fiber $\pi^{-1}(t)$ is the image of a $k$-disk for $K$.
\end{defn}

The following is the main result of \cite{Sch20} about the existence of $p$-rational disk-like global surface of section on $(L(p, 1),\xi_0)$.
\begin{lemma} \cite[Theorem 1.5]{Sch20}
Let $\lambda$ be a dynamically convex contact form on $(L(p, 1), \xi_0)$. Then its Reeb flow admits a $p$-unknotted closed Reeb orbit $P$ which is the binding of a rational open book decomposition. Each page of the open book is a rational disk-like global surface of section. Moreover, the Conley-Zehnder index of the $p$-th iterate of $P$ is 3.
\end{lemma}

As its corollary, we obtain the existence of $g_{p,1}$-symmetric periodic orbit on
tight three-sphere endowed with a contactomorphism $g_{p,1}$, which bounds a disk-like global surface of section.
\begin{corollary}\label{cor: existence of global suface of section}
Let $(S^3, \lambda, \xi_0)$ be a dynamically convex tight three-sphere satisfying
$g_{p,1}^*\lambda = \lambda$ for some $p \geq1$. Then there exists a $g_{p,1}$-symmetric periodic orbit on
$S^3$ which bounds a disk-like global surface of section $\mathfrak{D}$. It is the lift of a periodic orbit on $(L(p, 1),\xi_0)$
that bounds a $p$-rational disk-like
global surface of section.
\end{corollary}

Combining this corollary and Theorem \ref{thm: refinement of Franks}, we get one of our main results in this section:
\begin{theorem}\label{thm: application 1}
Let $(S^3, \lambda, \xi_0)$ be a dynamically convex tight three-sphere satisfying
$g_{p,1}^*\lambda = \lambda$ for some $p \geq1$. Then there exist two
or infinitely many $g_{p,1}$-symmetric periodic orbits on $S^3$.
\end{theorem}
\begin{proof}
Since $g_{p,1}^*\lambda = \lambda$, then the Reeb vector field $X_\lambda$
satisfies $g_{p,1}^*X_\lambda =X_\lambda$, from which we find that the
Reeb flow $\phi^t_{X_\lambda}$ satisfies $\phi^t_{X_\lambda}=g_{p,1}^{-1}\circ\phi^t_{X_\lambda}\circ g_{p,1}$, i.e., $g_{p,1}\circ\phi^t_{X_\lambda}=\phi^t_{X_\lambda}\circ g_{p,1}$.
We define a Poincar$\acute{e}$ $\frac{1}{p}$-return map
$\psi: \mathfrak{D}\setminus \partial\mathfrak{D}\rightarrow g_{p,1}(\mathfrak{D}\setminus \partial\mathfrak{D})$ by
$\psi(x)=\phi^{\tau(x)}_{X_\lambda}(x)$, where $\tau(x) := \min{\{t > 0 \mid\phi^t_{X_\lambda}(x)\in g_{p,1}(\mathfrak{D}\setminus \partial\mathfrak{D})\}}$.
Let $f=g_{p,1}^{-1}\circ\psi$. Then, $f$ preserves the area induced by $d\lambda|_{\mathfrak{D}\setminus \partial \mathfrak{D}}$, and the total area of $\mathfrak{D}\setminus \partial \mathfrak{D}$ is finite
\footnote{A page of the rational open book decomposition in \cite[Theorem 1.5]{Sch20} is in fact   the projection of a finite  energy holomorphic plane into the contact manifold (see \cite{Sch20} and \cite{HS16}).
  Similar to the second paragraph in page 265 of  \cite{HWZ98}, we have $\psi^* d \lambda |_{\mathfrak{D}\setminus \partial\mathfrak{D}}=d \lambda |_{g_{p,1}(\mathfrak{D}\setminus \partial\mathfrak{D})}$,  and
      $d \lambda |_{\mathfrak{D}\setminus \partial\mathfrak{D}}$ induces an area of $ \mathfrak{D}\setminus \partial\mathfrak{D}$ with finite total area.
}.
By Brouwer Plane Translation Theorem, $f$ has at least one interior fixed point $x_0$, then $\psi(x_0)=g_{p,1}(x_0)$ which corresponds to
a $g_{p,1}$-symmetric periodic orbit on $S^3$. Now, $f$ is also an area-preserving homeomorphism on
the open annuls $\mathbb{A}=\mathfrak{D}\setminus \{x_0, \partial\mathfrak{D}\}$, applying Theorem \ref{thm: refinement of Franks} by taking $n_0=p$,
we obtain that $$\sharp\left\{\bigcup_{(k,p)=1}\mathrm{Per}_{k}(f)\right\}=0~ \text{or}~ +\infty.$$
In the following, we prove that any $x\in\mathrm{Per}_{k}(f)$ corresponds to a $g_{p,1}$-symmetric periodic orbit on $S^3$,
where $(k,p)=1$. In fact, since $(k,p)=1$, then there exist some $l,q\in\mathbb{Z}$ such that $kq+pl=1$ which implies\
$f^{1-pl}(x)=f^{kq}(x)=x$. Note that $f=g_{p,1}^{-1}\circ\psi$, $g_{p,1}\circ\psi=\psi\circ g_{p,1}$
and $g_{p,1}^p=id$, then $\psi^{1-pl}=g_{p,1}(x)$ and $x$ corresponds to a $g_{p,1}$-symmetric periodic orbit on $S^3$,
the proof of the converse is the same. Hence there exist two
or infinitely many $g_{p,1}$-symmetric periodic orbits on $S^3$.
\end{proof}

For real tight three-sphere $(S^3, \lambda,\varrho)$ endowed with a contactomorphism $g_{p,1}$,
Kim obtains the existence of $(\varrho, g_{p,1})$-symmetric periodic orbit on
$S^3$ which bounds a $\varrho$-invariant disk-like global surface of section.
\begin{lemma} \cite[Theorem 1.12]{Kim}\label{lem: existence of symetric global suface of section}
Let $(S^3, \lambda,\varrho)$ be a dynamically convex real tight three-sphere such
that $\varrho$ is of the form $(z_1, z_2) \mapsto (e^{i\vartheta_1}\bar{z}_1, e^{i\vartheta_2}\bar{z}_2)$ with either
$\vartheta_1=0$ or $\vartheta_2=0$. Assume that $\lambda$ satisfies
$g_{p,1}^*\lambda = \lambda$ for some $p \geq1$. Then there exists a $(\varrho, g_{p,1})$-symmetric periodic orbit on
$S^3$ which bounds a $\varrho$-invariant disk-like global surface of section $\mathfrak{D}$. It is the lift of a
$\bar{\varrho}$-symmetric periodic orbit on $L(p, 1)$ that bounds a $\bar{\varrho}$-invariant $p$-rational disk-like
global surface of section.
\end{lemma}
Similar to Theorem \ref{thm: application 1}, we can get the following result:
\begin{theorem}\label{thm: application 2}
Let $(S^3, \lambda,\varrho)$ be a dynamically convex real tight three-sphere such
that $\varrho$ is of the form $(z_1, z_2) \mapsto (e^{i\vartheta_1}\bar{z}_1, e^{i\vartheta_2}\bar{z}_2)$ with either
$\vartheta_1=0$ or $\vartheta_2=0$. Assume that $\lambda$ satisfies
$g_{p,1}^*\lambda = \lambda$ for some $p \geq1$. Then, there exist two
or infinitely many $(\varrho, g_{p,1})$-symmetric periodic orbits on $S^3$.
\end{theorem}
\begin{proof}
We give the proof based on the proof of Theorem \ref{thm: application 1}. Since $\varrho^*\lambda =- \lambda$, then the Reeb vector field ${X_\lambda}$
satisfies $\varrho^*{X_\lambda} =-{X_\lambda}$, from which we find $\phi^t_{X_\lambda}=\varrho\circ\phi^{-t}_{X_\lambda}\circ \varrho$.
Note that $g_{p,1}\circ\varrho\circ g_{p,1}=\varrho$ and  $g_{p,1}\circ\phi^t_{X_\lambda}=\phi^t_{X_\lambda}\circ g_{p,1}$, we obtain
\begin{align*}
f\circ\varrho&=g_{p,1}^{-1}\circ\psi\circ\varrho=g_{p,1}^{-1}\circ\varrho\circ\psi^{-1}\nonumber\\
&=\varrho\circ g_{p,1}\circ\psi^{-1}=\varrho\circ\psi^{-1}\circ g_{p,1}\nonumber\\
&=\varrho\circ f^{-1},\nonumber
\end{align*}
hence the area-preserving homeomorphism $f$ on  $\mathfrak{D}\setminus \partial\mathfrak{D}$ is reversible with respect to the involution
$\varrho|_{\mathfrak{D}}$. It follows from Corollary 1.2 of \cite{Kang} that $f$ has at least one interior symmetric fixed point $x_0$, i.e.,
$f(x_0)=x_0=\varrho(x_0)$, then $\psi(x_0)=g_{p,1}(x_0)$ which corresponds to
a $(\varrho, g_{p,1})$-symmetric periodic orbit on $S^3$. Now, $f$ is also an area-preserving reversible homeomorphism on
the open annuls $\mathbb{A}=\mathfrak{D}\setminus \{x_0, \partial\mathfrak{D}\}$, applying Theorem \ref{thm: Franks thm for reversible map},
we obtain that if $f$ has a $k$-periodic point satisfying $(k,p)=1$, then $f$ has infinitely many symmetric periodic points
with period $k^\prime$ satisfying $(k^\prime,p)=1$. Similar to the proof in Theorem \ref{thm: application 1}, any symmetric $k^\prime$-periodic point $x$ corresponds to
a $g_{p,1}$-symmetric periodic orbit on $S^3$, and $f^l(x)=\varrho(x)$ for some $l\in\mathbb{N}$ which yields
$\psi^l(x)=g_{p,1}^l\circ \varrho(x)$. Then, $x$ corresponds to
a $(\varrho, g_{p,1})$-symmetric periodic orbit. Hence there exist two
or infinitely many $(\varrho, g_{p,1})$-symmetric periodic orbits on $S^3$.
\end{proof}

\begin{remark}
For the case $p=2$ of Theorem \ref{thm: application 2}, Kim has proved it  using   \cite[Theorem 1.3]{Kang} (see \cite[Theorem 1.12]{Kim}).
For dynamically convex real tight three-sphere $(S^3, \lambda,\varrho)$, Frauenfelder and Kang proved that there exist two or infinitely many $\varrho$-symmetric periodic orbits by \cite[Theorem 1.1]{Kang} (see\cite{FK}).
\end{remark}

\begin{remark}\label{remark: dynamically convex condition}
Hofer, Wysocki and Zehnder's dichotomy
theorem of \cite{HWZ98} is expected to hold for any contact form on contact three-manifold.
It is conjectured that in fact, every contact form on a closed connected three-manifold has either two or infinitely many periodic orbits.  This was recently proved  in~\cite{CDR20} for contact forms that are nondegenerate, extending a result of
\cite{CGHP} for non-degenerate torsion contact forms. Based on \cite{CGHP} and \cite{CDR20}, it is possible to remove the dynamically convexity condition in our Theorem \ref{thm: application 1} and Theorem \ref{thm: application 2}. Note also that Hryniewicz, Momin and Salom$\tilde{a}$o \cite{HMS15}
proved a remarkable Poincar\'e-Birkhoff type theorem for tight Reeb flows on $S^3$ where a global surface of section might not be available. Umberto Hryniewicz
has suggested to the authors in private correspondence that perhaps for our Theorem 1 there holds a general analogous result about Reeb flows on the tight $S^3$ possessing a Hopf link of periodic orbits, as in their Poincar\'e-Birkhoff theorem for Reeb flows of \cite{HMS15}.
\end{remark}

Our Theorem \ref{thm: application 1} and Theorem \ref{thm: application 2} can be applied to problems in celestial mechanics, for example,
the H$\acute{e}$non-Heiles system \cite{HH64} which describes chaotic motions
of a star in a galaxy with an axis of symmetry. The Hamiltonian is given by
\[H(q_1, q_2, p_1, p_2) =\frac{1}{2}(p_1^2 +p_2^2+q_1^2 +q_2^2)+ q_1^2q_2-\frac{1}{3}q^3_2, ~~~(q_1, q_2, p_1, p_2) \in\mathbb{R}^4.\]
It is invariant under the exact anti-symplectic involution
\[\varrho(q_1, q_2, p_1, p_2) = (-q_1, q_2, p_1,-p_2)\]
and the 3-periodic exact symplectic involution
\[\sigma(q, p) := (e^{2\pi i/3}q, e^{2\pi i/3}p),~~~ (q, p) = (q_1 + iq_2, p_1 + ip_2)\]
satisfying $\sigma\circ\varrho\circ\sigma=\varrho$.
There exist exactly two critical levels 0 and $\frac{1}{6}$ of $H$, and for every $c \in (0, \frac{1}{6})$,
the energy level set $H^{-1}(c)$ contains a unique bounded component $\Sigma_c$ that is diffeomorphic
to $S^3$. It is proved by Schneider that $\Sigma_c$ bounds a strictly convex domain
in $\mathbb{R}^4$ containing the origin \cite[Theorem 2.2]{Sch20}. Due to Theorem 3.7 of \cite{HWZ98}, it shows that the contact form $\lambda = \lambda_0|_{\Sigma_c}$ is dynamically
convex, where $\lambda_0$ is the Liouville 1-form on $\mathbb{R}^4$.

Note that $\sigma$ differs from $g_{3,1}$. Therefore, the induced contact structure $\widehat{\xi}$ on $L(3,1)\equiv S^3/ \mathbb{Z}^3$ is different from
$\xi_0$. Indeed, $(S^3/ \mathbb{Z}^3, \widehat{\xi})$ is contactomorphic to $(L(3,2), \xi_0)$ (cf., Section 2 of \cite{Sch20}). Thus, Theorems \ref{thm: application 1} and \ref{thm: application 2} can not apply directly to the H$\acute{e}$non-Heiles system, but the (invariant) 3-rational
disk-like global surface of section on $(S^3/ \mathbb{Z}^3, \widehat{\xi})$ still exists as we shall explain in Remark 7, which together with the proofs of
Theorems \ref{thm: application 1} and \ref{thm: application 2} still yield the applications on the H$\acute{e}$non-Heiles system:

\begin{theorem}
For every $c \in (0, \frac{1}{6})$, there exist two
or infinitely many $\sigma$-symmetric periodic orbits on $\Sigma_c$. Furthermore, there exist two
or infinitely many $(\varrho, \sigma)$-symmetric periodic orbits on $\Sigma_c$.
\end{theorem}
\begin{remark}
In \cite{CPR79}, Churchill, Pecelli, and
Rod find eight periodic orbits on $\Sigma_c$, labeled by $\Pi_1,\cdots ,\Pi_8$. They are related by
$\sigma^2(\Pi_j ) = \sigma(\Pi_{j+1}) = \Pi_{j+2}, j = 1, 4$, and $\Pi_7$ and $\Pi_8$ are $(\varrho, \sigma)$-symmetric periodic orbits.
In fact, by Section 2 of \cite{Sch20}, $\Pi_7$ and $\Pi_8$ descend to
3-unknotted closed Reeb orbits $\widehat{\Pi}_7$ and $\widehat{\Pi}_8$ on $(L(3, 2),\xi_0)$, which are invariant under $\varrho$. Moreover, their rational
self-linking numbers are $-1/3$. This together with a theorem of  Hryniewicz and Salom$\tilde{a}$o guarantees that each
$\widehat{\Pi}_i$, $i = 7, 8$ is the binding of a rational open book decomposition whose pages
are rational disk-like global surfaces of section for the Reeb flow on $\Sigma_c/ \mathbb{Z}_3\equiv
(L(3, 2), \xi_0)$ (cf. \cite[Corollary 1.8]{HS16}, \cite[Theorem 2.3]{Sch20}, and  \cite[Theorem 1.10]{Kim}), these rational disk-like global surfaces of section are invariant under $\varrho$. Then a similar argument of Corollary \ref{cor: existence of global suface of section} and Theorem \ref{thm: application 1}
shows that there exist two or infinitely many $\sigma$-symmetric periodic orbits on $\Sigma_c$, and a similar argument of Lemma \ref{lem: existence of symetric global suface of section} and Theorem \ref{thm: application 2} yields that there exist two or infinitely many $(\varrho, \sigma)$-symmetric periodic orbits on $\Sigma_c$.
\end{remark}

\end{document}